\documentclass[11pt]{amsart}
\usepackage[dvipsnames]{xcolor}
\usepackage{amssymb,amsmath,amsthm,enumerate,latexsym, graphics,shapepar,nccmath,mathtools}
\usepackage{mathptmx}
\usepackage{tikz-cd} 
\usepackage[all,2cell,dvips]{xy}
\usepackage[utf8]{inputenc}
\usepackage{hyperref}
\hypersetup{%
    colorlinks = true,
    linkcolor = BrickRed,
    citecolor = Green,
    urlcolor   = blue,
    filecolor = red,
}
\usepackage{cleveref}
\usepackage{enumitem}
\usepackage[margin=0.9in]{geometry}

\usepackage{parskip}
\usepackage{commands}

\begin{document}

\title{Linear quotients of connected ideals of graphs}

\author{H. Ananthnarayan}
\address{Department of Mathematics, I.I.T. Bombay, Powai, Mumbai~400076, Maharashtra, India}
\email{ananth@math.iitb.ac.in}

\author{Omkar Javadekar}
\address{Department of Mathematics, I.I.T. Bombay, Powai, Mumbai~400076, Maharashtra, India}
\email{omkar@math.iitb.ac.in}

\author{Aryaman Maithani}
\address{Department of Mathematics, University of Utah, 155 South 1400 East, Salt Lake City, UT~84112, USA}
\email{maithani@math.utah.edu}



\begin{abstract}
    As a higher analogue of the edge ideal of a graph, we study the $t$-connected ideal $\conn_{t}$. 
    This is the monomial ideal generated by the connected subsets of size $t$. 
    For chordal graphs, we show that $\conn_{t}$ has a linear resolution iff the graph is $t$-gap-free, and that this is equivalent to having linear quotients. 
    We then show that if $G$ is any gap-free and $t$-claw-free graph, then $\conn_{t}(G)$ has linear quotients and, hence, linear resolution.
\end{abstract}

\maketitle

\section{Introduction}

Given a finite simple graph $G$, an object that is of interest to study is its edge ideal $\edge(G)$. Being a square-free monomial ideal, $\edge(G)$ appears as the \emph{Stanley-Reisner ideal} of a simplicial complex, namely the \emph{independence complex} $\Ind(G)$. 
As such, determining algebraic properties of this ideal in terms of the combinatorial properties of the graph (or the complex) is an active area of research. One such property is the ideal having a linear resolution (\Cref{defn:linear-resolution-regularity}). 
In 1988, Fr\"{o}berg~\cite{Froberg} completely characterised such graphs. These are precisely the graphs whose complements are chordal. 

A next generalisation of the edge ideal is the $t$-path ideal $\edge_{t}(G)$, for $t \ge 2$. 
Banerjee~\cite{Banerjee} showed that if $G$ is gap-free and claw-free, then $\edge_{t}(G)$ has a linear resolution if ${3 \le t \le 6}$; 
if further $G$ is whiskered-$K_{4}$-free, then $\edge_{t}(G)$ has a linear resolution for all $t \ge 3$. 
This paper is in a similar direction. 
Instead of the $t$-path ideal, we look at the \emph{$t$-connected} ideal (\Cref{defn:t-connected-ideal}), which we denote $\conn_{t}(G)$. 
This is the ideal generated by the monomials corresponding to the $t$-connected subsets of $G$. 
In particular, $\conn_{t}(G) \supset \edge_{t}(G)$ for all $t \ge 2$ with equality if $t = 2, 3$. 
If $G$ is claw-free, then equality holds for $t = 4, 5$ as well. 
Consequently, we recover Banerjee's result for ${3 \le t \le 5}$ by obtaining the stronger result that $\edge_{t}(G)$ has linear quotients (see \Cref{cor:gap-free-claw-free-small-path-ideal-linear-quotients}).

We look at the notions of being \emph{$t$-gap-free} and \emph{$t$-claw-free}. These generalise the usual notions of being gap-free and claw-free, and are not stronger than them. More precisely, we have
\begin{align*}
    \text{gap-free $\Leftrightarrow$ $2$-gap-free $\Rightarrow$ $3$-gap-free $\Rightarrow$ $4$-gap-free $\Rightarrow$ $\cdots$} & \\
    \text{claw-free $\Leftrightarrow$ $3$-claw-free $\Rightarrow$ $4$-claw-free $\Rightarrow$ $5$-claw-free $\Rightarrow$ $\cdots$} &
\end{align*}

We show (\Cref{thm:chordal-LQ-LR-gap-equivalent}) that for a chordal graph $G$ and $t \ge 2$, $\conn_{t}(G)$ has a linear resolution iff $G$ is $t$-gap-free. 
This fulfils the goal of describing a(n algebraic) property of the ideal purely in terms of the (combinatorial) structure of the graph. 
This characterisation does not hold in general (\Cref{thm:failure-for-cycles}). 
We then show (\Cref{thm:gap-free-t-claw-free-LQ}) that for \emph{all} $t \ge 3$, if $G$ is gap-free and $t$-claw-free, then $\conn_{t}(G)$ has a linear resolution, which is similar in spirit to Banerjee's result. 

In~\cite{DeshpandeRoySinghTuyl}, the authors look at these connected ideals, where they partially generalise Fr\"{o}berg's result by showing that if $G$ is co-chordal, then $\conn_{t}(G)$ is \emph{vertex splittable} for all $t \ge 2$. 
We note that being vertex splittable is a stronger condition than having linear quotients. 
At the same time, being co-chordal is also a stronger condition than being gap-free. 
There are no implications between co-chordal and ($t$-)claw-free.

In this paper, we prove that the desired ideals have linear resolutions by proving that they have \emph{linear quotients} (\Cref{defn:linear-quotients}). This term were introduced in~\cite{HerzogTakayama}. 
This definition already appeared in a different form in~\cite{BatziesWelker}, where they called such ideals \emph{shellable} and showed that such ideals have a minimal free Lyubeznik resolution. 
In~\cite[Theorem 2.7]{JahanZheng}, it was shown that such ideals have componentwise linear resolution. 
In our setup, our ideals will be equigenerated, so linear quotients would imply linear resolutions. In contrast to having a linear resolution, the property of a monomial ideal having linear quotients is independent of the field (\Cref{rem:linear-quotients-independent-field}). 
The same is not true for the property of having a linear resolution. A typical example is the Stanley-Reisner ideal of the triangulation of $\mathbb{R}\mathbb{P}^{2}$, as mentioned in~\cite{Reisner}. 
This ideal has a linear resolution precisely if the characteristic of the underlying field is not $2$. However, we note that for monomial ideals generated in degree two, linear quotients and linear resolutions are equivalent, see~\cite[Theorem 3.2]{HerzogHibiZheng}. 
In particular, Fr\"{o}berg's theorem completely characterises which graphs have edge ideals having linear quotients.

While we do not use any theory of simplicial complexes, we draw the connections here. 
For a graph $G$ and $r \ge 1$, the \deff{$r$-independence complex} $\Ind_{r}(G)$ of $G$ is defined to be the collection of subsets $C \subset V(G)$ such that each connected component of the induced subgraph $G[C]$ has at most $r$ vertices. 
This is a simplicial complex, and $\conn_{t}(G)$ is precisely the Stanley-Reisner ideal of $\Ind_{t - 1}(G)$. 
In particular, $\conn_{t}(G)$ having linear quotients is equivalent to the dual complex $\Ind_{t - 1}(G)^{\vee}$ being shellable. 
In~\cite{AbdelmalekDeshpandeGoyalRoySingh} it was shown that $\Ind_{r}(T)$ is shellable for all trees $T$ and all $r \ge 1$. 
As trees form a subclass of chordal graphs, we have answered precisely the question of when the dual is shellable (\Cref{thm:chordal-LQ-LR-gap-equivalent}). 
Note that $\Ind_{1}(G) = \Ind(G)$, so these ideals are natural generalisations of the edge ideal.

The paper is organised as follows. 
In \Cref{sec:preliminaries}, we introduce the relevant preliminaries on graph theory and graded resolutions. 
In particular, we define linear resolutions and linear quotients, and remark that the latter notion is field independent. 
In \Cref{sec:t-gap-free-graphs}, we define the term \emph{$t$-gap-free}, and deduce properties of gap-free graphs. 
In \Cref{sec:t-connected-deals}, we define our ideal $\conn_{t}$ of study and note $t$-gap-free being a necessary condition for these ideals to have linear resolutions. 
In \Cref{sec:linear-quotients-for-chordal}, we show that this necessary condition is sufficient for chordal graphs. 
In \Cref{sec:linear-quotients-gap-free-t-claw-free}, we define the notion of a graph being $t$-claw-free, and prove that gap-free and $t$-claw free imply linear resolution of $\conn_{t}$ for $t \ge 3$. 
In \Cref{sec:further-questions}, we note some questions that were motivated by the computational results.

\textbf{Acknowledgements.} We thank P. Deshpande and A. Singh for organising the NCM workshop \emph{Cohen Macaulay simplicial complexes in graph theory (2023)} in CMI\@. 
This workshop introduced us to this field and provided an environment for fruitful discussions. 
We also thank S. Selvaraja for drawing our attention to~\cite[Theorem 1.4]{HaWoodroofe} which led to us proving the converse for chordal graphs. 
We are grateful to the anonymous referee for proofreading the paper carefully and for pointing out the precise result known in literature about claw-free and gap-free graphs having linear resolutions, see~\cite[Theorem 1.1]{Banerjee} and \Cref{cor:gap-free-claw-free-small-path-ideal-linear-quotients}. 
Several examples and conjectures were tested by the computer algebra systems \texttt{SageMath}~\cite{sagemath} and \texttt{Macaulay2}~\cite{M2}, and the package \texttt{nauty}~\cite{nauty}; the use of these is gratefully acknowledged. 

\section{Preliminaries} \label{sec:preliminaries}

\subsection{Graph Theory}

A \emph{graph} $G$ is an ordered tuple of finite sets $(V(G), E(G))$ such that $E(G)$ is some collection of subsets of $V(G)$ of size exactly two. The elements of $V(G)$ are called the \deff{vertices} of $G$, and elements of $E(G)$ the \deff{edges}. \newline
Given a set $C \subset V(G)$, we denote by $G[C]$ the induced subgraph on $C$. We say that $C$ is \deff{connected} if $G[C]$ is a connected graph. If further $\md{C} = t$, then we say that $C$ is \deff{$t$-connected}. \newline
Given a vertex $v \in V(G)$, we define $N(v) \coloneqq \{w \in V(G) : \{v, w\} \in E(G)\}$, i.e., $N(v)$ is the set of \deff{neighbours} of $v$. 
A vertex $v$ is called \deff{isolated} if $N(v) = \emptyset$, and a \deff{leaf} if $\md{N(v)} = 1$. \newline
Given a subset $C \subset V(G)$, we define $N(C) \coloneqq \bigcup_{v \in C} N(v) \setminus C$. 

For $t \ge 3$, $K_{1, t}$ denotes the graph with vertex set $\{0, 1, \ldots, t\}$ and edge set $\{\{0, i\} : 1 \le i \le t\}$. The \deff{center} of $K_{1, t}$ is the unique vertex which is not a leaf. 
A \deff{tree} is a connected graph with no cycles. 
A \deff{chordal graph} is a graph that contains no induced subgraph isomorphic to a cycle on four or more vertices. 

\begin{lem} \label{lem:at-least-two-non-cut-vertices}
    Let $G$ be a connected graph with $\md{V(G)} \ge 2$. \newline 
    Then, the set $\{v \in V(G) : G \setminus v \text{ is connected}\}$ has cardinality at least two.
\end{lem}
\begin{proof} 
    Pass to a spanning tree and pick two leaves.
\end{proof}

\begin{lem} \label{lem:connected-disjoint-union}
    Let $C \subset V(G)$ be connected. If $C = A \sqcup B$ with $A$, $B$ nonempty, then there exist vertices $a \in A$ and $b \in B$ such that $\{a, b\}$ is an edge.
\end{lem}
\begin{proof} 
    Pick any $x \in A$ and $y \in B$. By hypothesis, there exists a path
    \begin{equation*} 
        x = v_{0} \to \cdots \to v_{n} = y
    \end{equation*}
    with $v_{i} \in C$. Since $v_{0} \in A$ and $v_{n} \in B$, there exists some $i$ such that $v_{i} \in A$ and $v_{i + 1} \in B$. These play the desired roles of $a$ and $b$.
\end{proof}

\subsection{Resolutions}

For any homogeneous ideal $I \subset S \coloneqq \mathbb{K}[x_{1}, \ldots, x_{n}]$, there exists a \deff{graded minimal free resolution} of $I$, i.e., an exact sequence
\begin{equation*} 
    0 \to \bigoplus_{j \in \mathbb{N}} S(-j)^{\beta_{n, j}(I)} \to \bigoplus_{j \in \mathbb{N}} S(-j)^{\beta_{n - 1, j}(I)} \to \cdots \to \bigoplus_{j \in \mathbb{N}} S(-j)^{\beta_{0}, j(I)} \to I \to 0.
\end{equation*}
The integer $\beta_{i, j}(I)$ is uniquely determined by $I$ and is called the \deff{$(i, j)$-th graded Betti number} of $I$. 

\begin{defn} \label{defn:linear-resolution-regularity}
    Suppose $d \ge 0$, and $I \subset S$ is a homogeneous ideal generated by its degree $t$ elements. We say that $I$ has a \deff{linear resolution} if $\beta_{i, j}(I) = 0$ for all $j \neq i + t$. The zero ideal is also said to have a linear resolution.

    The \deff{(Castelnuovo-Mumford) regularity} of $I$ is defined to be $\reg(I) \coloneqq \max\{j - i : \beta_{i, j}(I) \neq 0\}$.
\end{defn}

\begin{rem} \label{rem:linear-resolution-and-regularity}
    Given a homogeneous ideal $I \subset S$ generated in degree $t$, $I$ has a linear resolution iff $\reg(I) = t$.
\end{rem}

\begin{rem} \label{rem:regularity-I-S-mod-I}
    The above notions can be defined more generally for graded modules, not necessarily generated in the same degree. In particular, it makes sense to talk about the Betti numbers, linear resolutions, and regularity of $S/I$, where $I$ is a homogeneous ideal. The minimal resolutions of $I$ and $S/I$ can be obtained from one another. We refer the reader to~\cite{PeevaBook} for an introduction to graded resolutions. We just remark that $\reg(S/I) = \reg(I) - 1$. 
\end{rem}

\subsection{Linear Quotients} \label{subsec:linear-quotients}

\begin{defn} \label{defn:linear-quotients}
    Let $\mathbb{K}$ be a field, $S = \mathbb{K}[x_{1}, \ldots, x_{n}]$ the polynomial ring in $n$ variables, and $I \subset S$ a monomial ideal. We denote by $\mingens(I)$ the unique minimal monomial system of generators of $I$. We say that $I$ has \deff{linear quotients}, if there exists an order ${u_{1} < \cdots < u_{m}}$ on $\mingens(I)$ such that the colon ideal $\langle u_{1}, \ldots, u_{i - 1} \rangle : \langle u_{i} \rangle$ is generated by a subset of the variables, for $i = 2, \ldots, m$. Any such order is said to be an \deff{admissible order}.

    We consider the zero ideal to also be a monomial ideal having linear quotients.
\end{defn} 

\begin{rem} \label{rem:linear-quotients-independent-field}
    While we talked about the ideal and colons inside a polynomial ring over $\mathbb{K}$, the property of having linear quotients is independent of $\mathbb{K}$. Indeed, given two monomials $u$, $v$, we can define the colon $u : v$ to be the monomial $\lcm(u, v)/v$. Then, we have
    \begin{align*} 
        \langle u \rangle : \langle v \rangle &= \langle u : v \rangle, \\
        \langle u_{1}, \ldots, u_{m} \rangle : \langle v \rangle &= \langle u_{1} : v, \ldots, u_{m} : v \rangle
    \end{align*}
    for monomials $u_{1}, \ldots, u_{m}, u, v \in S$. Recall that a monomial $v$ is in a monomial ideal $I$ iff $v$ is divisible by some element of $\mingens(I)$. This shows that the property of a monomial ideal being generated by a subset of the variables is independent of the coefficient field. In turn, $I$ having linear quotients depends only on the set of monomials $\mingens(I)$ and not the base field $\mathbb{K}$. In particular, when discussing any of the various monomials ideals associated to graphs, we do not have to mention the base field when talking about linear quotients.
\end{rem}

\section{\texorpdfstring{$t$}{t}-gap-free graphs} \label{sec:t-gap-free-graphs}

A \deff{hypergraph} $\mathcal{H}$ is a tuple $(V(\mathcal{H}), \mathcal{E}(\mathcal{H}))$, where $V(\mathcal{H})$ is a finite set (whose elements are called \deff{vertices}), and $\mathcal{E}(\mathcal{H})$ is a subset of the power set of $V(\mathcal{H})$ (whose elements are called \deff{hyperedges}). We shall assume that every hyperedge has the same (nonzero) cardinality. In particular, there is no containment among distinct hyperedges. \newline
A \deff{matching} in $\mathcal{H}$ is a subset $\mathcal{M} \subset \mathcal{E}$ such that any two distinct elements of $\mathcal{M}$ are disjoint. A matching $\mathcal{M} = \{E_{1}, \ldots, E_{s}\}$ is said to be an \deff{induced matching} if for any $E \in \mathcal{E}$, we have
\begin{equation*} 
    E \subset E_{1} \sqcup \cdots \sqcup E_{s} \quad\Rightarrow\quad E = E_{i} \text{ for some $i$}.
\end{equation*}

We define the \deff{induced matching number} of $\mathcal{H}$ as 
\begin{equation*} 
    \gamma(\mathcal{H}) \coloneqq \max\{\md{\mathcal{M}} : \mathcal{M} \text{ is an induced matching in $\mathcal{H}$}\}.
\end{equation*}

Given a graph $G$ and an integer $t \ge 2$, we associate to it the hypergraph $\mathcal{H} \coloneqq \mathcal{H}(G, t)$ as follows: 
\begin{align*} 
    V(\mathcal{H}) &\coloneqq V(G), \\
    \mathcal{E}(\mathcal{H}) &\coloneqq \{C \subset V(G) : C \text{ is $t$-connected}\}.
\end{align*}
Recall that $t$-connected means that $G[C]$ is connected and $\md{C} = t$.

\begin{prop} \label{prop:gamma-H-one-equivalent}
    Let $G$ be a graph, $t \ge 2$, and $\mathcal{H} = \mathcal{H}(G, t)$. The following are equivalent:
    \begin{enumerate}[label=(\alph*)]
        \item $\gamma(\mathcal{H}) \le 1$.
        \item Given any two disjoint connected sets $C, C' \subset V(G)$ of cardinality $t$, there exists $c \in C$ and $c' \in C$ such that $\{c, c'\} \in E(G)$.
    \end{enumerate}
\end{prop}
\begin{proof} 
    (a) $\Rightarrow$ (b): Let $C, C'$ be as stated. Note that $\{C, C'\}$ is a matching in $H$. Since $\gamma(H) \le 1$, there is a third hyperedge $C''$ contained in $C \cup C'$. Pick vertices $x \in C \cap C''$ and $y \in C' \cap C''$. Since $C''$ is connected, there is a path 
    \begin{equation*} 
        x = x_{0} \to \cdots \to x_{n} = y
    \end{equation*}
    with each $x_{i} \in C'' \subset C \cup C'$. Since $x \in C$ and $y \in C'$, there is some $i$ such that $x_{i} \in C$ and $x_{i + 1} \in C'$. These are the desired $c$ and $c'$.

    (b) $\Rightarrow$ (a): Suppose $\{C, C'\}$ is a matching. We show that this is not an induced matching. By assumption, there is an edge $\{c, c'\} \in E(G)$ for some $c \in C$, $c' \in C'$. Since $t \ge 2$, \Cref{lem:at-least-two-non-cut-vertices} lets us pick a vertex $x \in C \setminus \{c\}$ such that $C \setminus \{x\}$ is connected. Now, $(C \setminus \{x\}) \sqcup \{c'\}$ is also connected, of cardinality $t$, contained in $C \cup C'$, and distinct from both $C$ and $C'$. 
\end{proof}

We give a name to the graphs satisfying the above condition.
\begin{defn}
    A graph $G$ is called \deff{$t$-gap-free} if $\gamma(\mathcal{H}(G, t)) \le 1$.
\end{defn}

We have the following chain of implications.
\begin{center}
    co-chordal $\Rightarrow$ $2$-gap-free $\Rightarrow$ $3$-gap-free $\Rightarrow$ $4$-gap-free $\Rightarrow$ $\cdots$.
\end{center}

Recall that graph $G$ is called \deff{gap-free} if $G$ is $2$-gap-free. 
In other words, if $e_{1} = \{a, b\}$ and $e_{2} = \{c, d\}$ are disjoint edges of $G$, then there is an edge connecting a vertex of $e_{1}$ with a vertex of $e_{2}$. 
In yet other words, $G$ contains no induced subgraph isomorphic to $P_{2} \sqcup P_{2}$. This formulation makes it clear that if $G$ is gap-free, then so is $G \setminus v$ for any $v \in V(G)$. 

\Cref{prop:gamma-H-one-equivalent} tells us that being $t$-gap-free generalises the above: If $C_{1}$ and $C_{2}$ are disjoint $t$-connected subsets of $G$, then there is an edge connecting a vertex of $C_{1}$ with a vertex of $C_{2}$. 

We note some connectivity properties of a gap-free graph $G$. 

\begin{prop} \label{prop:gap-free-union-connected-is-connected}
    Let $G$ be a gap-free graph. Let $C_{1}, \ldots, C_{n} \subset V(G)$ be connected subsets with $\md{C_{i}} \ge 2$ for all $i$. Then, $\bigcup_{i} C_{i}$ is connected.
\end{prop}
\begin{proof} 
    By induction, we may assume $n = 2$. If $C_{1} \cap C_{2}$ is nonempty, then the result is true (without any gap-free hypothesis). \newline
    Thus, we may assume that $C_{1}$ and $C_{2}$ are disjoint. Since their cardinalities are at least two, there exist (necessarily disjoint) edges $e_{1} \subset C_{1}$ and $e_{2} \subset C_{2}$. Since $G$ is gap-free, we get an edge between them.
\end{proof}

\begin{cor} \label{cor:edge-from-C-to-outside}
    Let $G$ be a gap-free graph, and let $C, C' \subset V(G)$ be connected subsets of size at least two. If $C' \setminus C$ is nonempty (i.e., $C' \not\subset C$), then there exist $v \in C$ and $w \in C' \setminus C$ such that $\{v, w\}$ is an edge.
\end{cor}
\begin{proof} 
    By \Cref{prop:gap-free-union-connected-is-connected}, $C \cup C'$ is connected. We can write
    \begin{equation*} 
        C \cup C' = C \sqcup (C' \setminus C).
    \end{equation*}
    By \Cref{lem:connected-disjoint-union}, the result follows.
\end{proof}

\begin{obs} \label{obs:gap-free-isolated-vertices}
    If $G$ is gap-free, then $V(G)$ can be written as a union $\{\ell_{1}\} \sqcup \cdots \sqcup \{\ell_{n}\} \sqcup C$, where each $\ell_{i}$ has no neighbours, $C$ is connected, and $\md{C} \neq 1$. In other words, there is at most one component which has size greater than $1$.

    More generally, any subset $X \subset V(G)$ can be written in the above form, i.e., as a disjoint union of singletons and a non-singleton set, with these being the connected components of $G[X]$. Now, if $C \subset V(G)$ is connected, and $a \in C$, then $C \setminus a$ can be decomposed as above. The singletons $\{\ell_{i}\}$ are precisely the leaves of $G[C]$ that are connected to $a$.
\end{obs}

\begin{prop} \label{prop:linear-quotients-combinatorial}
    Let $X$ be a finite set and $k \ge 1$. \newline
    There exists a total order $<$ on $\{C \subset X : \md{X} = k\}$ satisfying the following: if $C' < C$, then there exists $C'' < C$ such that $C'' \setminus C = \{x\}$ and $x \in C'$.
\end{prop}
\begin{proof} 
    The lexicographic order works. 

    More precisely: to every $k$-element set, we associate a $k$-tuple by listing the elements in ascending order. That is,

    \begin{equation*} 
        \{a_{1}, \ldots, a_{k}\} \leftrightarrow (a_{1}, \ldots, a_{k})
    \end{equation*}
    with $a_{1} < \cdots < a_{k}$. Thus, we may describe an order on the set of (strictly) increasing $k$-tuples and check the corresponding property for this order.

    The order is defined as: For distinct tuples $\mathbf{a} = (a_{1}, \ldots, a_{k})$ and $\mathbf{b} = (b_{1}, \ldots, b_{k})$, let $i \in \{1, \ldots, k\}$ be the smallest index for which $a_{i} \neq b_{i}$. We set $\mathbf{b} < \mathbf{a}$ if $b_{i} < a_{i}$. 

    We now check that this order has the desired property. To this end, let $\mathbf{b} < \mathbf{a}$ be increasing $k$-tuples. As before, let $i$ be the smallest index where $\mathbf{a}$ and $\mathbf{b}$ differ. We can then write
    \begin{align*} 
        \mathbf{a} = (x_{1}, \ldots, x_{i - 1},\; & a_{i},\; a_{i + 1}, \ldots, a_{k}), \\
        \mathbf{b} = (x_{1}, \ldots, x_{i - 1},\; & b_{i},\; b_{i + 1}, \ldots, b_{k}),
    \end{align*}
    with $b_{i} < a_{i}$. This implies that the tuple 
    \begin{equation*} 
        \mathbf{c} = (x_{1}, \ldots, x_{i - 1},\; b_{i},\; a_{i + 1}, \ldots, a_{k})
    \end{equation*}
    is strictly increasing and we have $\mathbf{c} < \mathbf{a}$. Moreover, $\mathbf{c} \setminus \mathbf{a} = \{b_{i}\}$ with $b_{i} \in \mathbf{b}$.
\end{proof}

\begin{rem}
    The above is equivalent to saying that $\conn_{k}(K_{n})$ has linear quotients for all $k, n \ge 1$.
\end{rem}

\section{\texorpdfstring{$t$}{t}-connected ideals} \label{sec:t-connected-deals}

\begin{defn} \label{defn:t-connected-ideal}
    Given a graph $G$ and a field $\mathbb{K}$, we define $\mathbb{K}[G]$ to be the polynomial ring over $\mathbb{K}$ with variables $\{x_{v} : v \in V(G)\}$. \newline
    Given $t \ge 2$, we define the \deff{$t$-connected ideal} $\conn_{t}(G)$ as
    \begin{equation*} 
        \conn_{t}(G) \coloneqq \langle x_{i_{1}} \cdots x_{i_{t}} : \{i_{1}, \ldots, i_{t}\} \subset G \text{ is connected} \rangle \subset \mathbb{K}[G].
    \end{equation*}
\end{defn}
For ease of notation, given a subset $C \subset V(G)$, we denote by $x_{C}$ the product $\prod_{c \in C} x_{c}$. Thus, $\conn_{t}(G) = \langle x_{C} : C \subset V(G) \text{ is $t$-connected}\rangle$. Equivalently, $\conn_{t}(G)$ is the edge ideal of $\mathcal{H}(G, t)$. 

Note that the edge ideal of $G$ is $\conn_{2}(G)$, $\conn_{3}(G)$ coincides with the path ideal $\edge_{3}(G)$, and $\conn_{t}(G) \supset \edge_{t}(G)$ for $t \ge 4$. This containment can be strict, as is witnessed by $K_{1, t}$. See however \Cref{lem:claw-free-path-and-connected-ideal-coincide}.

We recall the following fact, which is a special case of~\cite[Corollary 3.9]{MoreyVillarreal}. The version we state below appears as~\cite[Theorem 1.4]{HaWoodroofe}.
\begin{thm}
    Let $\mathcal{H}$ be a hypergraph with edge ideal $I \subset S$, and suppose that every hyperedge of $\mathcal{H}$ has cardinality $t$. Then, $\reg(S/I) \ge (t - 1) \gamma(\mathcal{H})$.
\end{thm}

Since $\conn_{t}(G)$ is the edge ideal of $\mathcal{H}(G, t)$ and generated by monomials of degree $t$, we immediately get (cf. \Crefrange{rem:linear-resolution-and-regularity}{rem:regularity-I-S-mod-I}) the following.

\begin{cor} \label{cor:linear-resolution-implies-t-gap-free}
    Let $G$ be a graph, and $t \ge 2$.
    \begin{equation*} 
        \conn_{t}(G) \text{ has a linear resolution} \Rightarrow \text{$G$ is $t$-gap-free}.
    \end{equation*}
\end{cor}
In the next section, we will prove the converse for chordal graphs. 

\begin{prop} \label{prop:induced-subgraph-regularity-conn-t}
    Let $H$ be an induced subgraph of $G$. Then, $\reg(\conn_{t}(H)) \le \reg(\conn_{t}(G))$.
\end{prop}
\begin{proof} 
    By definition, we have $\conn_{t}(H) = \edge(\mathcal{H}(H, t))$ and $\conn_{t}(G) = \edge(\mathcal{H}(G, t))$, where $\edge(\mathcal{H})$ denotes the edge ideal of a hypergraph $\mathcal{H}$. \newline
    By~\cite[Lemma 2.5]{HaWoodroofe}, it suffices to show that $\mathcal{H}(H, t)$ is an induced subhypergraph of $\mathcal{H}(G, t)$. This follows at once since if $C \subset V(H)$ is any subset, then $H[C] \cong G[C]$ as $H$ is an induced subgraph.
\end{proof}

\section{Linear quotients for chordal graphs} \label{sec:linear-quotients-for-chordal}

\begin{thm} \label{thm:chordal-LQ-LR-gap-equivalent}
    Let $G$ be a chordal graph, and $t \ge 2$. The following are equivalent:
    \begin{enumerate}[label=(\alph*)]
        \item $\conn_{t}(G)$ has linear quotients.
        \item $\conn_{t}(G)$ has a linear resolution.
        \item $G$ is $t$-gap-free.
    \end{enumerate}
\end{thm}
The above can be seen as a generalisation of Fr\"{o}berg's theorem for chordal graphs.
\begin{proof} 
    Only (c) $\Rightarrow$ (a) is to be shown, as we do now. 
    We prove this by induction on $\md{V(G)}$. 
    If $\md{V(G)} < t$, then $\conn_{t}(G)$ is the zero ideal, and the result is clear. 

    Assume now that $\md{V(G)} \ge t$. 
    Pick a vertex $v \in V(G)$ such that the neighbours of $v$ form a clique, i.e., the induced subgraph $G[N(v)]$ is complete; 
    the existence of such a vertex follows from the fact that $G$ is chordal, see~\cite{Rose}.
    Thus $G \setminus v$ is a chordal graph on fewer vertices, and the corresponding hypergraph continues to have induced matching number one. 
    By induction, there is an admissible order $<$ on the generators of $\conn_{t}(G \setminus v)$. 
    Now, note that 
    \begin{equation*} 
        \mingens(\conn_{t}(G)) = \mingens(\conn_{t}(G \setminus v)) \sqcup \{x_{C} : v \in C,\, C \text{ is $t$-connected}\}.
    \end{equation*}
    We now show that appending the monomials from the latter set to $<$ gives us an admissible order, proving the result. 
    (The order that we give to the latter set does not matter.)

    Let $C$ be a $t$-connected subset containing $v$, and $D$ a $t$-connected subset such that $D < C$. 
    We wish to show the existence of a subset $C' < C$ such that $x_{C'} : x_{C}$ is a variable that divides $x_{D}$. 

    \textbf{Claim 1.} The set $N(C) \cap D$ is nonempty, i.e., there is a vertex $w \in D \setminus C$ that is a neighbour of a vertex in $C$.
    
    \begin{proof} 
        If $C$ and $D$ are not disjoint, then this follows by applying \Cref{lem:connected-disjoint-union} to the decomposition of $D$ given by $(C \cap D) \sqcup (D \setminus C)$. 
        If $C$ and $D$ are disjoint, then this follows from \Cref{prop:gamma-H-one-equivalent}. 
    \end{proof}

    We now fix such a vertex $w \in N(C) \cap D$.

    \textbf{Claim 2.} There exists $u \in C \setminus \{v\}$ such that $\{u, w\}$ is an edge.

    \begin{proof} 
        By construction, $w$ is a neighbour of some $u \in C$. 
        If $u \neq v$, then we are done. 
        Assume now that $u = v$. 
        Because $\md{C} = t \ge 2$, we see that $v$ has a neighbour $u' \in C$. 
        By hypothesis, the neighbours of $v$ in $G$ form a complete graph and hence, $\{u', w\}$ is an edge. 
    \end{proof}

    Thus, we now have a vertex $w \in D \setminus C$ that is connected to a vertex $u \in C \setminus \{v\}$. 
    In particular, the set $C' \coloneqq (C \setminus \{v\}) \sqcup \{w\}$ has size $t$. 
    Moreover, note that because the neighbours of $v$ form a complete graph, the induced graph on $C \setminus \{v\}$ is connected and contains $u$. 
    Because $\{u, w\}$ is an edge, we see that $C'$ is connected as well. 
    Thus, we have constructed a $t$-connected set $C' < C$ such that $x_{C'} : x_{C}$ divides $x_{D}$, as desired.
\end{proof}

If $G$ is not chordal, then the equivalence in \Cref{thm:chordal-LQ-LR-gap-equivalent} does not hold. 
It fails even for cycles, as the next result shows.

\begin{thm} \label{thm:failure-for-cycles}
    Let $t \ge 2$. Then, $C_{2t + 1}$ is $t$-gap-free but $\conn_{t}(C_{2t + 1})$ does not have linear resolution. \newline
    If $t \ge 3$, then $C_{2t}$ is $t$-gap-free but $\conn_{t}(C_{2t})$ does not have linear resolution. \newline
    In particular, these ideals do not have linear quotients.
\end{thm}
\begin{proof} 
    In either case, the statement about being $t$-gap-free is straightforward to verify. We note that for cycles, our ideal $\conn_{t}$ coincides with the usual path ideal. By~\cite[Corollary 5.5]{AlilooeeFaridi}, we have
    \begin{align*} 
        \reg(\conn_{t}(C_{2t})) = 2t - 2, \\
        \reg(\conn_{t}(C_{2t + 1})) = 2t - 1.
    \end{align*}
    Both the quantities above are $> t$ under our hypotheses. Since our ideals are generated in degree $t$, this means that the resolutions are not linear (see \Cref{rem:linear-resolution-and-regularity}).
\end{proof}

\begin{prop} \label{prop:failure-for-induced-cycles}
    If $G$ contains an induced $n$-cycle, then $\conn_{t}(G)$ does not have linear resolution if $n > t + 2$.
\end{prop}
\begin{proof} 
    By \Cref{prop:induced-subgraph-regularity-conn-t}, it suffices to show that $\reg(S/\conn_{t}(C_{n})) > t - 1$, where $S \coloneqq \mathbb{K}[G]$. Using the division algorithm, write
    \begin{equation*} 
        n = p(t + 1) + d
    \end{equation*}
    for some $p \ge 0$ and $0 \le d \le t$. Since $n > t + 2$, we get that $p \ge 1$. Moreover, if $p = 1$, then we must have $d > 1$. By~\cite[Corollary 5.5]{AlilooeeFaridi}, we know that
    \begin{equation*} 
        \reg(R/\conn_{t}(C_{n})) = 
        \begin{cases}
            (t - 1)p + d - 1 & \text{if } d \neq 0, \\
            (t - 1)p & \text{if } d = 0.
        \end{cases}
    \end{equation*}
    Since $p \ge 1$, the above can be equal to $t - 1$ only if $p = 1$ and $d \in \{0, 1\}$. We have already ruled out this possibility.
\end{proof}

\section{Linear quotients for gap-free and \texorpdfstring{$t$}{t}-claw-free graphs} \label{sec:linear-quotients-gap-free-t-claw-free}

\begin{defn}
    Let $t \ge 3$. A graph $G$ is called \deff{$t$-claw-free} if $G$ contains no induced subgraph isomorphic to $K_{1, t}$.
\end{defn} 
Note that the usual notion of claw-free coincides with $3$-claw-free. Moreover, we have
\begin{center}
    $3$-claw-free $\Rightarrow$ $4$-claw-free $\Rightarrow$ $5$-claw-free $\Rightarrow$ $\cdots$.
\end{center}

\begin{thm} \label{thm:gap-free-t-claw-free-LQ}
    Let $t \ge 3$ be an integer. Suppose $G$ is a gap-free and $t$-claw-free graph. Then, $\conn_{t}(G)$ has linear quotients. In particular, $\conn_{t}(G)$ has a linear resolution.
\end{thm}
\begin{proof} 
    Let $t$ be as given. We prove the statement by induction on number of vertices of $G$. If $\md{V(G)} < t$, the statement is clear for then $\conn_{t}(G) = \langle 0 \rangle$. 

    Let $G$ be a graph with $\md{V(G)} \ge t$.\footnote{Note that we do not assume $G$ to be connected. So it may be still possible that $\conn_{t}(G)$ is the zero ideal.} Pick any vertex $a \in V(G)$. Then, $G \setminus a$ is again gap-free and $t$-claw-free and hence, $\conn_{t}(G \setminus a)$ has linear quotients. 

    As in the proof of \Cref{thm:chordal-LQ-LR-gap-equivalent}, it suffices to specify an appropriate order on the $t$-connected subsets $C$ containing $a$. We set up some notations first. 

    Let $C$ be a $t$-connected set containing $a$. By $L(C)$ we denote the set of leaves of $G[C]$ that are neighbours of $a$. (Equivalently, these are the isolated vertices of $C \setminus \{a\}$; see \Cref{obs:gap-free-isolated-vertices}.) We set $B(C) \coloneqq C \setminus \left(\{a\} \cup L(C)\right)$, calling this the \emph{branch} of $C$ (with respect to $a$). By \Cref{obs:gap-free-isolated-vertices}, we know that either $\md{B(C)} = 0$ or $\md{B(C)} \ge 2$.

    Let $\mathcal{C}$ denote the set of all $t$-connected subsets that contain $a$. We first partition $\mathcal{C}$ into $\mathcal{C}_{0}, \ldots, \mathcal{C}_{t - 1}$ as
    \begin{equation*} 
    \mathcal{C}_{k} \coloneqq \{C \in \mathcal{C} : \md{L(C)} = k\}.
    \end{equation*}
    We set $\mathcal{C}_{-1} \coloneqq \mingens(\conn_{t}(G \setminus a))$.

    Any arbitrary ordering can be put on $\mathcal{C}_{0}$. For $k \ge 1$, we order the sets based on their branches. For each $B \subset V(G)$ that is of the form $B(C)$ for some $C \in \mathcal{C}_{k}$, we fix one order $<_{B}$ on the $k$-element subsets of $N(a) \setminus B$ as given by \Cref{prop:linear-quotients-combinatorial}. If $C, C' \in \mathcal{C}_{k}$ are distinct such that $B = B(C) = B(C')$, then $L(C) \neq L(C')$. We set $C < C'$ if $L(C) <_{B} L(C')$. (We only require such an order among the sets having the same branch. The order among the different branches is immaterial as the proof will show.)  \newline
    Finally, we set
    \begin{equation*} 
        \mingens(\conn_{t}(G \setminus a)) = \mathcal{C}_{-1} < \mathcal{C}_{0} < \cdots < \mathcal{C}_{t - 1}
    \end{equation*}
    
    This defines the ordering. Given $C \in \mathcal{C}$, we now check that $J \coloneqq \langle x_{C'} : C' < C \rangle : \langle x_{C} \rangle$ is generated by variables. We break this into three cases, depending on $k \coloneqq \md{L(C)}$. In each case, we will show that if $C \in C_{k}$ and $C' < C$, then there exists a $C'' \in \mathcal{C_{\ell}}$ with $\ell < k$ having the property that $C'' \setminus C = \{v\}$ for some $v \in C'$. In particular, this tells us that $C'' < C$ and that the variable $x_{v}$ divides the colon $x_{C'} : x_{C}$.

    \textbf{Case 1}. $k = 0$. In this case, note that $C \setminus \{a\}$ continues to remain connected. The same is true for any $C' < C$, regardless of whether $a \in C'$ or not. Thus, by \Cref{cor:edge-from-C-to-outside} applied to $C \setminus \{a\}$ and $C' \setminus \{a\}$, there exists $b \in C \setminus \{a\}$ and $v \in C' \setminus C$ such that $\{b, v\}$ is an edge. \newline
    But then $C'' \coloneqq (C \setminus \{a\}) \sqcup \{v\}$ is a $t$-connected set not containing $a$. Since $a \notin C''$, we see that $C''$ necessarily comes before $C$. Thus, we get the variable $x_{v}$ using the colon $x_{C''} : x_{C}$, which divides $x_{C'} : x_{C}$.

    \textbf{Case 2.} $1 \le k < t - 1$. Equivalently, $\md{B(C)} \neq 0$. As noted above, this implies that $\md{B(C)} \ge 2$. (In turn, this means that the case $k = t - 2$ is nonexistent.) In particular, $B(C)$ contains an edge.

    \textbf{Claim 1.} If $v \in V(G) \setminus C$ is a neighbour of some $w \in C \setminus \{a\}$, then the variable $x_{v}$ is in the colon $J$.
    \begin{proof} 
        If $w \in B(C)$, then pick any $\ell \in L(C)$ and consider $C'' \coloneqq (C \setminus \{\ell\}) \sqcup \{v\}$. The subset $C''$ is $t$-connected and has $\md{L(C'')} < \md{L(C)}$, showing $C'' < C$. Thus, $x_{v} \in J$.

        If $w \notin B(C)$, then $w \in L(C)$. Let $\{b, c\} \subset B(C)$ be an edge. Note that $\{v, w\}$ are $\{b, c\}$ are disjoint. Since $G$ is gap-free, there is an edge between those sets. Necessarily, this edge cannot involve $w$ (since the only neighbour of $w$ in $C$ is $a$). Thus, $v$ is a neighbour of some vertex in $B(C)$, and we are in the previous case.
    \end{proof}

    Armed with the claim, let us assume that $C' < C$ is a $t$-connected subset. If $C' \setminus C$ contains a neighbour of some vertex in $C \setminus \{a\}$, then we are done by Claim 1. Suppose that this is not the case.

    \textbf{Claim 2.} $a \in C'$. 
    \begin{proof} 
        By \Cref{cor:edge-from-C-to-outside}, there exists an edge connecting $B(C)$ to some $v \in C' \setminus B(C)$. By assumption, we must have $v \in C$. Thus, $v \in L(C) \sqcup \{a\}$. Being leaves, the elements of $L(C)$ cannot be neighbours to any vertex in $B(C)$. Thus, $a = v \in C'$.
    \end{proof}

    Thus, $C' \in \mathcal{C}$. Since $C' < C$, we must have $\md{L(C')} \le \md{L(C)}$. Consequently, $\md{B(C')} \ge \md{B(C)} \ge 2$.

    \textbf{Claim 3.} $B(C) = B(C')$.
    \begin{proof} 
        If $B(C') \not\subset B(C)$, then \Cref{cor:edge-from-C-to-outside} would give us an edge from some $b \in B(C)$ to some $v \in B(C') \setminus B(C)$. As in the proof of the previous claim, we get that $v = a$. But this is a contradiction since $a \notin B(C')$. Thus, $B(C') \subset B(C)$. Checking cardinalities, we conclude equality.
    \end{proof}

    Now, since $B(C) = B(C') \eqqcolon B$ and $C' < C$, we must have $L(C') <_{B} L(C)$. By the definition of $<_{B}$, there exists a $k$-element subset $L'' \subset N(a) \setminus B$ such that $L '' < L(C)$ and $L'' \setminus L(C) = \{\ell\} \subset L(C')$. Consider the set
    \begin{equation*} 
        C'' = B \sqcup \{a\} \sqcup L''.
    \end{equation*}
    $C''$ is $t$-connected since $B \cup \{a\}$ is connected and each element of $L''$ is a neighbour of $a$. Note that $L(C'') \subset L''$. If this containment is proper, then $C'' < C$ since then $\md{L(C'')} < k$. If $L(C'') = L''$, then $B(C'') = B$. By definition of the order among subsets having the same $B$, we again get $C'' < C$. \newline
    Now, the colon $x_{C''} : x_{C}$ gives us the variable $x_{\ell}$, which divides the colon $x_{C'} : x_{C}$.
    
    \textbf{Case 3.} $k = t - 1$, i.e., $G[C]$ is isomorphic to $K_{1, t - 1}$, with $a$ as the center of the claw. Let $C'$ be a $t$-connected subset. By \Cref{cor:edge-from-C-to-outside}, there is an edge $\{v, w\}$ with $v \in C' \setminus C$ and $w \in C$. If $w = \ell \in L(C)$, then we are done by considering $C'' = (C \setminus \{\ell'\}) \sqcup \{v\}$, where $\ell' \in L(C)$ is a leaf different from $\ell$. \newline
    Thus, we may assume $w = a$. Now, consider the induced subgraph $G[C \sqcup \{v\}]$. This has vertex set $V' \coloneqq C \sqcup \{v\}$ and $a \in V'$ is a neighbour to all the $t$ elements of $V' \setminus \{a\}$. Since $G$ is $t$-claw-free, there must be an edge within $L(C) \sqcup \{v\}$. But there cannot be an edge between two elements of $L(C)$. Thus, $v$ is connected to some leaf $\ell \in L(C)$ and we are back in the previous case.
\end{proof}

\begin{cor} \label{cor:gap-free-claw-free-small-path-ideal-linear-quotients}
    If $G$ is a gap-free and claw-free graph, then the $t$-path ideal $\edge_{t}(G)$ has linear quotients for $t \in \{3, 4, 5\}$. In particular, the ideal has a linear resolution.
\end{cor}
The above ideals having linear resolutions was first proven in~\cite[Theorem 1.1]{Banerjee}.
\begin{proof} 
    Let $G$ and $t$ be as in the statement. Then, $G$ is $t$-claw-free. Thus, $\conn_{t}(G)$ has linear quotients by \Cref{thm:gap-free-t-claw-free-LQ}. The result now follows from \Cref{lem:claw-free-path-and-connected-ideal-coincide}.
\end{proof}

\begin{lem} \label{lem:claw-free-path-and-connected-ideal-coincide}
    If $G$ is a claw-free graph, then $\edge_{t}(G) = \conn_{t}(G)$ for $t \in \{3, 4, 5\}$. That is, the $t$-path ideal and $t$-connected ideal coincide.
\end{lem}
\begin{proof} 
    The inclusion $\edge_{t}(G) \subset \conn_{t}(G)$ always holds since paths are connected. We prove the reverse inclusion for $t \in \{3, 4, 5\}$. 
    Assume that $G$ is a claw-free graph. 
    It suffices to prove that whenever $C \subset V(G)$ is $t$-connected, then $C$ contains a path on $t$ vertices. 
    Since each connected graph possesses a spanning tree, it suffices to prove the following: 
    if $T$ is a (possibly non-induced) subgraph of $G$ that is a tree on $t$ vertices, then the vertices of $T$ form a path in $G$ in some order. Note that this path need not be a path in $T$.

    Case $t = 3$: There is exactly one tree on $3$ vertices, namely the $3$-path. Thus, we are done. 

    Case $t = 4$: There are exactly two trees on $4$ vertices, drawn below.

    \begin{center}
        \fourpathsymb \qquad \clawsymb{4}
    \end{center}

    The first graph is a $4$-path and we are done. 
    For the second graph, note that since $G$ is claw-free, there is an additional edge present in $G$ that connects two leaves in the above tree. 
    Thus, $G$ contains the following (possibly non-induced) subgraph.

    \begin{center}
        \clawcompletedsymb{4}
    \end{center}

    The above contains a $4$-path, as indicated by the red edges.

    Case $t = 5$: There are exactly three trees on $5$ vertices, drawn below.
    \begin{center}
        \fivepathsymb \qquad \fivetreenotpathsymb \qquad \clawsymb{5}
    \end{center}

    Arguing as before, we see that $G$ being claw-free implies the existence of enough extra edges that we get a path of length $5$.
\end{proof}

\begin{rem}
    \Cref{lem:claw-free-path-and-connected-ideal-coincide} cannot be extended to $t = 6$. Indeed, the following graph $G$---called the \deff{net graph}---is gap-free and claw-free but $\edge_{6}(G) = \langle 0 \rangle \neq \langle x_{V(G)} \rangle = \conn_{6}(G)$. 

    \begin{center}
        \sunletsymb{3}
    \end{center}

    In fact, one can check that the above graph is the sole hindrance for $t = 6$. More precisely, if $G$ is a claw-free graph that contains no induced subgraph isomorphic to the net graph, then $\conn_{6}(G) = \edge_{6}(G)$.

    However, the conclusion of \Cref{cor:gap-free-claw-free-small-path-ideal-linear-quotients} is still true for $t = 6$ by~\cite[Theorem 1.1]{Banerjee}, namely that $\edge_{6}(G)$ has a linear resolution if $G$ is a gap-free and claw-free graph. 
\end{rem}

\section{Further questions} \label{sec:further-questions}

\begin{question}
    Can the ``$t$-claw-free graph'' hypothesis in \Cref{thm:gap-free-t-claw-free-LQ} be dropped? In other words, if $G$ is gap-free, then does $\conn_{t}(G)$ have linear quotients for all $t \ge 3$?
\end{question}
Note that ``gap-free'' cannot be dropped, as witnessed by \Cref{thm:failure-for-cycles}. Moreover, the above is not true for $t = 2$. Indeed, one can take any gap-free graph that is not co-chordal; for example, $C_{5}$. 

\begin{question}
    For $t \ge 3$, what are all the graphs $G$ for which $\conn_{t}(G)$ has linear quotients?
\end{question}
\Cref{cor:linear-resolution-implies-t-gap-free} tells us that $G$ must be $t$-gap-free, and \Cref{prop:failure-for-induced-cycles} tells us that $G$ cannot contain a large induced cycle.

We recall the definition of \emph{vertex splittable} ideals from~\cite[Definition 2.1]{MoradiKhoshAhang}.
\begin{defn}
    A monomial ideal $I \subset \mathbb{K}[X]$ is called \deff{vertex splittable} if it is obtained in the following recursive manner.
    \begin{enumerate}[label=(\alph*)]
        \item $(0)$ is vertex splittable. Any principal monomial ideal is vertex splittable.
        \item If there is a variable $x \in X$ and vertex splittable ideals $J, K$ of $\mathbb{K}[X \setminus \{x\}]$ such that 
        \begin{equation*} 
            I = (xJ + K)\mathbb{K}[X], \quad K \subset J, \quad \text{and} \quad \mingens(I) = \mingens(xJ) \sqcup \mingens(K),
        \end{equation*}
        then $I$ is vertex splittable.
    \end{enumerate}
\end{defn}

The following theorem shows how these ideals fit into the picture with the existing notions. For proofs, see~\cite[Theorems 2.3, 2.4]{MoradiKhoshAhang}.

\begin{thm}
    For monomial ideals, the following chain of implications hold:
    \begin{center}
        Vertex splittable $\Rightarrow$ Linear Quotients $\Rightarrow$ Linear Resolution.
    \end{center}

    A simplicial complex $\Delta$ is \emph{vertex decomposable} (\cite[Definition 1.1]{MoradiKhoshAhang}) iff $I_{\Delta^{\vee}}$ is vertex splittable.
\end{thm}

\begin{question}
    Suppose $t$ and $G$ satisfy the hypotheses of either \Cref{thm:chordal-LQ-LR-gap-equivalent} or \Cref{thm:gap-free-t-claw-free-LQ}, i.e., one of the following is true.
    \begin{itemize}
        \item $t \ge 2$ and $G$ is a $t$-gap-free chordal graph.
        \item $t \ge 3$ and $G$ is gap-free and $t$-claw-free.
    \end{itemize}
    Is it true that $\conn_{t}(G)$ is vertex splittable?
\end{question}
The case $t = 2$ and $G$ is a (2-)gap-free chordal graph is answered by~\cite{DeshpandeRoySinghTuyl} since then $G$ is co-chordal. 

\printbibliography
\end{document}